\documentclass[11 pt]{amsart}
\usepackage[T1]{fontenc}
\usepackage[utf8]{inputenc}
\usepackage[english]{babel}
\usepackage{mathtools}
\usepackage{amsfonts}
\usepackage{amsthm}
\usepackage{amsmath,amstext,amsfonts,verbatim}
\usepackage{xcolor}
\usepackage{amssymb}
\definecolor{darkgreen}{rgb}{0.0, 0.5, 0.0}
\newcommand{\C}{\mathbb{C}}
\newcommand{\T}{\mathbb{T}}
\newcommand{\N}{\mathbb{N}}
\newcommand{\p}{\mathbb{P}}
\newcommand{\A}{\mathcal{A}}
\newcommand{\D}{\mathbb{D}}
\newcommand{\E}{\mathbb{E}}
\newcommand{\h}{\mathrm{H}}

\newcommand{\R}{\mathbb{R}}
\numberwithin{equation}{section}
\theoremstyle{plain}
\newtheorem{theo}{Theorem}[section]
\newtheorem{prop}[theo]{Proposition}
\newtheorem{lemma}[theo]{Lemma}
\newtheorem{coro}[theo]{Corollary}

\newtheorem{question}{Question}
\theoremstyle{definition}
\newtheorem{rem}[theo]{Remark}
\title{Random interpolating sequences in the polydisc and the unit ball}
\author{Alberto Dayan, Brett D. Wick and Shengkun Wu}

\address{Alberto Dayan, Department of Mathematics and Statistics, Washington University in St. Louis, MO 63130, USA}
\email{ alberto.dayan@wustl.edu}
\address{Brett D. Wick, Department of Mathematics and Statistics, Washington University in St. Louis, MO 63130, USA}
\email{ wick@math.wustl.edu}
\address{Shengkun Wu, College of Mathematics and Statistics, Chongqing University, Chongqing, 401331, PR China}
\email{shengkunwu@foxmail.com}

\thanks{Alberto Dayan was partially supported by National Science Foundation Grant DMS 1565243}
 \thanks{Brett D. Wick is partially supported by National Science Foundation DMS grant 1800057, and by Australian Research Council -- DP 190100970. }
 \thanks{Shengkun Wu is supported by CSC201906050022.}

\date{\today}

\begin{document}
\maketitle
\begin{abstract}
We study almost sure separating and interpolating properties of random sequences in the polydisc and the unit ball. In the unit ball, we obtain the 0-1 Komolgorov law for a sequence to be interpolating almost surely for all the Besov-Sobolev spaces $B_{2}^{\sigma}\left(\mathbb{B}_{d}\right)$, in the range $0 < \sigma\leq1 / 2$. For those spaces, such interpolating sequences coincide with interpolating sequences for their multiplier algebras, thanks to the Pick property. This is not the case for the Hardy space $\h^2(\D^d)$ and its multiplier algebra $\h^\infty(\D^d)$: in the polydisc, we obtain a sufficient and a necessary condition for a sequence to be $\h^\infty(\D^d)$-interpolating almost surely. Those two conditions do not coincide, due to the fact that the deterministic starting point is less descriptive of interpolating sequences than its counterpart for the unit ball. On the other hand, we give the $0-1$ law for random interpolating sequences for $\h^2(\D^d)$.
\end{abstract}

\section{Introduction}
 A sequence $Z=(z_n)_{n\in\N}$ in  the unit disc $\D$ is interpolating for $\h^\infty$ if, given any bounded sequence $(w_n)_{n\in\N}$ in $\C$ there exists a bounded analytic function $f$ on $\D$ so that $f(z_n)=w_n$, for any $n$ in $\N$. The celebrated work of Carleson, \cite{carloss} and \cite{carlocm}, characterized interpolating sequences in term of separation properties. To be precise, let
\[
b_\tau(z):=\frac{\tau-z}{1-\overline{\tau}z},\qquad z\in\D
\]
be the involutive Blaschke factor at $\tau$ in $\D$, and let, for any $z$ and $w$ in $\D$,
\[
\rho(z, w):=\left|b_z(w)\right|
\] be the \emph{pseudo-hyperbolic distance} in $\D$.  $Z$ is
\begin{itemize}
\item \emph{weakly separated} if
\[
\inf_{n\ne k}\rho(z_n, z_k)>0;
\]
\item \emph{uniformly separated}  if
\[
\inf_{n\in\N}\prod_{k\ne n}\rho(z_n, z_k)>0.
\]
\end{itemize}
Carleson proved in \cite{carloss} that $Z$ is interpolating if and only if it is uniformly separated. Later on, \cite{carlocm}, he characterized uniform separation in terms of a measure theoretic condition and weak separation:
\begin{theo}[Carleson]
\label{theo:carlo}
A sequence $Z$ in $\D$ is uniformly separated if and only if it is weakly separated and the measure
\[
\mu_Z:=\sum_{n\in\N}(1-|z_n|^2)\delta_{z_n}
\]
is a Carleson measure for $\h^2(\D)$.
\end{theo}
Throughout this note, a measure $\mu$  on a domain $D$ will be a Carleson measure for a reproducing kernel Hilbert space $\mathcal{H}_k$ of holomorphic functions on $D$  if
\[
||f||_{\mathrm{L}^2(D, \mu)}\leq C||f||_{\mathcal{H}_k}\qquad f\in\mathcal{H}_k,
\]
for some $C>0$.  Later sections will take $D=\D^d$, the unit polydisc, or $D=\mathbb{B}^d$, the unit ball, respectively: the kernels that we are going to choose for such domains are the Szegö kernel on the polydisc and the Besov-Sobolev kernels on the unit ball.

In certain instances, the randomization of the conditions studied by Carleson become more tractable and provide insight into the structure of interpolating sequences.  Cochran studied in \cite{coch} separation properties of random sequences. A random sequence in the unit disc is defined as follows: let $(\theta_n)_{n\in\N}$ be a sequence of independent random variables, all distributed uniformly in $(0, 2\pi)$ and defined on the same probability space $(\Omega, \A, \p)$. Then, for any choice of a deterministic sequence of radii $(r_n)_{n\in\N}$ approaching $1$ define
\[
\lambda_n(\omega):=r_ne^{i\theta_n(\omega)},\qquad\omega\in\Omega.
\]
Considering the random sequence $\Lambda(\omega)=(\lambda_n(\omega))_{n\in\N}$, the 0-1 Kolmogorov law yields that events such as
\[
\begin{split}
&\mathcal{W}:=\{\Lambda \text{ is weakly separated}\}\\
&\mathcal{U}:=\{\Lambda \text{ is uniformly separated}\}\\
&\mathcal{C}:=\{\mu_\Lambda \text{ is a Carleson measure for $\h^2(\D)$}\}\\
&\mathcal{I}:=\{\Lambda \text{ is an interpolating sequence}\}
\end{split}
\]
have probability zero or one, thanks to the independence of the arguments of the points in $\Lambda$. Let
\begin{equation}
\label{eqn:shells}
I_j:=\{z\in\D:  1-2^{-j}\leq |z|<1-2^{-(j+1)}\}\qquad j\in\N
\end{equation}
 be the $j$th dyadic annulus of $\D$, and let
\begin{equation}
\label{eqn:N}
N_j:= \# \Lambda\cap I_j.
\end{equation}
All the randomness of the sequence is on the arguments of the points in $\Lambda$, and therefore $(N_j)_{j\in\N}$ is a deterministic sequence. Cochran proved in \cite[Th. 2]{coch} that $\p(\mathcal{W})=1$ provided that
\begin{equation}
\label{eqn:sum:finite}
\sum_{j\in\N} N_j^22^{-j}<\infty,
\end{equation}
and that $\p(\mathcal{W})=0$ whenever the sum in \eqref{eqn:sum:finite} diverges. Later on, Rudowicz showed in \cite{rudo} that \eqref{eqn:sum:finite} is a sufficient condition for $\mu_\Lambda$ to be a Carleson measure for $\h^2(\D)$ almost surely, and concluded, thanks to Theorem \ref{theo:carlo}, that $\p(\mathcal{I})=1$ if and only if \eqref{eqn:sum:finite} holds. In particular, condition \eqref{eqn:sum:finite} encodes all those random sequences so that $\mathcal{W}$, $\mathcal{U}$ and $\mathcal{I}$ have all probability one.

The goal of this paper is to study random interpolating sequences on the polydisc and the $d$ dimensional unit ball. A sequence $Z=(z_n)_{n\in\N}$ in $\D^d$ is interpolating for $\h^\infty(\D^d)$ if, given any bounded $(w_n)_{n\in\N}$ in $\C$ there exists a bounded holomorphic function $f$ on $\D^d$ so that $f(z_n)=w_n$, for all $n$. On the polydisc, the deterministic starting point is the following (partial) analogous of Carleson interpolation Theorem for sequences in the polydisc \cite{bern}:
\begin{theo}[Berndtsson, Chang and Lin]
\label{theo:poly:interp}
Let $Z=(z_n)_{n\in\N}$ be a sequence in $\D^d$, and let \textnormal{(a)}, \textnormal{(b)} and \textnormal{(c)} denote the following statements:
\begin{itemize}
\item[(a)]
\begin{equation}
\label{eqn:ss}
\inf_{n\in\N}\prod_{k\ne n}\rho_G(z_n, z_k)>0;
\end{equation}
\item[(b)] $Z$ is interpolating for $\h^\infty(\D^d)$;
\item[(c)] The measure
\[
\mu_Z:=\sum_{n\in\N}\left(\prod_{i=1}^d(1-|z_n^i|^2)\right)\delta_{z_n}
\]
is a Carleson measure for $\h^2(\D^d)$ and
\begin{equation}
\label{eqn:ws}
\inf_{n\ne k}\rho_G(z_n, z_k )>0.
\end{equation}
\end{itemize}
Then \textnormal{(a)}$\implies$\textnormal{(b)}$\implies$\textnormal{(c)}, and none of the converse implications hold.
\end{theo}
Conditions \eqref{eqn:ss} and \eqref{eqn:ws} are separation conditions, both stated in terms of the so called \emph{Gleason distance} on the polydisc:
\[
\rho_G(w, z):=\max_{i=1,\dots d}\rho(z^i, w^i)\qquad z, w\in\D^d.
\]
Throughout this note, \eqref{eqn:ss} will refer to \emph{uniform separation} on the polydisc, while \eqref{eqn:ws} defines a \emph{weakly separated} sequence on the polydisc.\\
Theorem \ref{theo:poly:interp} represents one of the best known attempts to characterize $\h^\infty(\D^d)$-interpolating sequences on the polydisc in terms of its hyperbolic geometry. One can find a characterization for interpolating sequences for bounded analytic functions on the bi-disc in \cite{agler}, stated in terms of uniform separation conditions on an entire class of reproducing kernels on $\D^2$. The motivation of the first part of this note is to find out whether condition (a),  and (c) of Theorem \ref{theo:poly:interp} are equivalent at least \emph{almost surely}. A negative answer would imply that Theorem \ref{theo:poly:interp} is far from being a characterization. A positive answer would give the 0-1 Kolmogorov law for $\h^\infty(\D^d)$-interpolating sequences in the polydisc with random arguments. The construction of a random sequence $\Lambda$ on the polydisc follows the same outline as for the case of the unit disc. Let $\T^d$ be the $d$-dimensional torus in $\C^d$, and let $(\theta^1_n, \dots, \theta^d_n)_{n\in\N}$ be a sequence of independent and indentically distributed random variables taking values on $\T^d$, all distributed uniformly and defined on the same probability space $(\Omega, \A, \p)$. Let $(r_n)_{n\in\N}$ be a sequence in $[0, 1)^d$, and define a random sequence $\Lambda=(\lambda_n)_{n\in\N}$ in $\D^d$ as
\[
\lambda_n(\omega)=\left(r^1_ne^{i\theta^1_n(\omega)}, \dots, r^d_ne^{i\theta^d_n(\omega)}\right), \qquad\omega\in\Omega.
\]
The events  of interest are going to be
\[
\begin{split}
&\mathcal{W}(\D^d):=\{\Lambda \text{ is weakly separated in $\D^d$}\}\\
&\mathcal{U}(\D^d):=\{\Lambda \text{ is uniformly separated in $\D^d$}\}\\
&\mathcal{C}(\h^2(\D^d)):=\{\mu_\Lambda \text{ is a Carleson measure for $\h^2(\D^d)$}\}\\
&\mathcal{I}(\D^d):=\{\Lambda \text{ is an interpolating sequence for $\h^\infty(\D^d)$}\}.
\end{split}
\]
Our first aim is to give necessary conditions and sufficient conditions for $\Lambda$ to be interpolating for $\h^\infty(\D^d)$ almost surely. This will be achieved by studying separately the probability of the events $\mathcal{W}(\D^d)$, $\mathcal{U}(\D^d)$ and $\mathcal{C}(\h^2(\D^ d))$, and by applying Theorem \ref{theo:poly:interp}. Looking for separation conditions on $(r_n)_{n\in\N}$ that yield almost sure separation properties for $\Lambda$, \eqref{eqn:shells} and \eqref{eqn:N} are extended to the $d$ dimensional case by considering
\begin{equation}
\label{eqn:rectangles}
I_m:=\{z\in\D^d: 1-2^{-m_i}\leq|z^i|<1-2^{-(m_i+1)}, i=1,\dots d\}
\end{equation}
and
\[
N_m= \# \Lambda\cap I_m,
\]
for any multi-index $m=(m_1,\dots,m_d)$ in $\N^d$. Throughout this note, $|m|=m_1+\dots+m_d$ will denote the length of $m$.

The first main result partially extends Cochran's and Rudowicz's works to the polydisc:
\begin{theo}
\label{theo:polydisc}
Let $\Lambda$ be a random sequence in $\D^d$. Then
\begin{itemize}
\item[(i)] If
\begin{equation}
\label{eqn:sum:polydisc}
\sum_{m\in\N^d}N_m^22^{-|m|}<\infty
\end{equation}
then $\p(\mathcal{W}(\D^d))=1$. If the sum in \eqref{eqn:sum:polydisc} diverges, then $\p(\mathcal{W}(\D^d))=0$.\\

\item[(ii)]If
\begin{equation}
\label{eqn:polydisc:strong}
\sum_{m\in\N^d}N_m^{{1+\frac{1}{d}}}2^{-\frac{|m|}{d}}<\infty
\end{equation}
then $\p(\mathcal{U}(\D^d))=1$.\\

\item[(iii)] If \eqref{eqn:sum:polydisc} holds, then $\p(\mathcal{C}(\h^2(\D^d)))=1$.
\end{itemize}
\end{theo}

Observe that the case $d=1$ yields  Rudowicz's and Cochran's characterization of random interpolating sequences on the unit disc. In general, part (i)  of the above Theorem gives the 0-1 Komolgorov law for a sequence to be weakly separated. In part (ii) and (iii), the result gives a sufficient condition for a sequence to be almost surely uniformly separated and to generate a Carleson measure for the Hardy space in the polydisc. In particular, thanks to Theorem \ref{theo:poly:interp}, it is the case that the 0-1 Kolmogorov law for almost surely interpolating sequences for $\h^\infty(\D^d)$ lies somewhere in between \eqref{eqn:polydisc:strong} and \eqref{eqn:sum:polydisc}:
\begin{coro}
Let $\Lambda$ be a random sequence on $\D^d$. Then
\begin{itemize}
\item[(i)] If \eqref{eqn:polydisc:strong} holds, then $\p(\mathcal{I}(\D^d))=1$;
\item[(ii)] If the sum in \eqref{eqn:sum:polydisc} diverges, then $\p(\mathcal{I}(\D^d))=0$.
\end{itemize}
\end{coro}
Proposition \ref{prop:example} will give an example of a class of random sequences for which the 0-1 Kolmogorov law for almost surely $\h^\infty(\D^d)$-interpolating sequences coincides with the sum in \eqref{eqn:sum:polydisc}. Whether this is the case for a general choice of the radii $(r_n)_{n\in\N}$ remain, for us, open. Nevertheless, we will observe in Section \ref{sec:hs} how \eqref{eqn:sum:polydisc} implies that the Szegö Grammian for a random sequence in the polydisc differs from the identity only by a Hilbert-Schmidt operator, a rather strong separation condition for the random kernel functions in the Hardy space associated to $\Lambda$. In particular, this will give the $0-1$ law for a random sequence $\Lambda$ to be interpolating for $\h^2(\D^d)$. In the deterministic setting, a sequence $(z_n)_{n\in\N}$ on $\D^d$ is interpolating for $\h^2(\D^d)$ if the map
\[
f\in\h^2(\D^d)\mapsto \left(\prod_{i=1}^ d\sqrt{1-|z_n^ i|^2}f(z_n)\right)_{n\in\N}\in l^2
\]
is surjective and bounded. This, in particular, is equivalent of asking that the Szegö Grammian associated to $(z_n)_{n\in\N}$ is bounded above and below.\\
Given a random sequence $\Lambda$ in $\D^d$, let
\[
\tilde{\mathcal{I}}(\D^d):=\{\Lambda\,\text{is interpolating for}\,\h^2(\D^d)\}.
\]
Any $\h^\infty(\D^d)$-interpolating sequence on $\D^d$ is also $\h^2(\D^d)$-interpolating, and the converse does not hold, since $\h^2(\D^d)$ has not the Pick property (for an example of a sequence which is $\h^2(\D^2)$-interpolating but not $\h^\infty(\D^2)$-interpolating, see \cite{amar}). Therefore, $\mathcal{I}(\D^d)\subseteq\tilde{\mathcal{I}}(\D^d)$. We show that $\tilde{\mathcal{I}}(\D^d)$ has the same $0-1$ law of $\mathcal{W}(\D^d)$:
\begin{theo}
\label{theo:intpolyh2}
Let $\Lambda$ be a random sequence in $\D^d$. Then
\[
\p(\tilde{\mathcal{I}}(\D^d))=\begin{cases}
0\quad&\text{if}\quad\sum_{m\in\N^d}N_m^22^{-|m|}=\infty\\
1\quad&\text{if}\quad\sum_{m\in\N^d}N_m^22^{-|m|}<\infty
\end{cases}.
\]
\end{theo}
\,
\\

Related questions about interpolation for function spaces on the unit ball in $\C^d$ are also considered.  The authors in \cite{wick2019} studied the interpolating sequences in the Dirichlet spaces over the unit disc and this serves as some of the motivation for the results in the ball.  Section \ref{sec:ball}, will generalize some theorems in \cite{wick2019} to the unit ball. Because the generalization of the Dirichlet spaces is the Besov-Sobolev spaces, random interpolating sequences in the Besov-Sobolev spaces $B_{2}^{\sigma}\left(\mathbb{B}_{d}\right)$ are studied, where $0<\sigma<\infty$. In \cite{Arcozzi}, a characterization of interpolating sequence in the Besov-Sobolev spaces in the case of $0<\sigma\leq\frac{1}{2}$ was given.  Because a characterization exists only in this range, that is the case focused upon in this paper.

Let $\mathbb{B}_{d}$ be the unit ball in $\mathbb{C}^{d}$. Let $d z$ be Lebesgue measure on $\mathbb{C}^{d}$ and let $d \lambda_{d}(z)=\left(1-|z|^{2}\right)^{-d-1} d z$ be the invariant measure on the ball. For an integer $m \geq 0,$ and for $0 < \sigma<\infty, 1<p<\infty, m+\sigma>d / p$ define the analytic Besov-Sobolev spaces $B_{p}^{\sigma}\left(\mathbb{B}_{d}\right)$ to consist of those holomorphic functions $f$ on the ball such that
$$\|f\|_{B_{p}^{\sigma}\left(\mathbb{B}_{d}\right)}^{p}=\left\{\sum_{k=0}^{m-1}\left|f^{(k)}(0)\right|^{p}
+\int_{\mathbb{B}_{d}}\left|\left(1-|z|^{2}\right)^{m+\sigma} f^{(m)}(z)\right|^{p} d \lambda_{d}(z)\right\}^{\frac{1}{p}}<\infty.$$
Here $f^{(m)}$ is the $m^{t h}$ order complex derivative of $f .$ The spaces $B_{p}^{\sigma}\left(\mathbb{B}_{d}\right)$ are independent of $m$ and are Banach spaces.
A Carleson measure for $B_{p}^{\sigma}\left(\mathbb{B}_{d}\right)$ is a positive measure defined on $\mathbb{B}_{d}$ such that the following Carleson embedding holds for $f \in B_{p}^{\sigma}\left(\mathbb{B}_{d}\right)$
$$
\int_{\mathbb{B}_{d}}|f(z)|^{p} d \mu \leq C_{\mu}\|f\|_{B_{p}^{\sigma}\left(\mathbb{B}_{d}\right)}^{p}.
$$
Given $\sigma$ with $0 < \sigma\leq1 / 2$ and a discrete set $Z=\left\{z_{i}\right\}_{i=1}^{\infty} \subset \mathbb{B}_{d}$ define the associated measure $\mu_{Z}=\sum_{j=1}^{\infty}\left(1-\left|z_{j}\right|^{2}\right)^{2 \sigma} \delta_{z_{j}} $.  $Z$ is an interpolating sequence for $B_{2}^{\sigma}\left(\mathbb{B}_{d}\right)$ if the restriction map $R$ defined by $R f\left(z_{i}\right)=f\left(z_{i}\right)$ for $z_{i} \in Z \operatorname{maps} B_{2}^{\sigma}\left(\mathbb{B}_{d}\right)$ into and onto $\ell^{2}\left(Z, \mu_{Z}\right) .$

\begin{theo}\label{interpolation}
Given $\sigma$ with $0 < \sigma\leq 1/2$ and $\mu_{Z}=\sum_{j=1}^{\infty}\left(1-\left|z_{j}\right|^{2}\right)^{2 \sigma} \delta_{z_{j}} .$ Then Z is an interpolating sequence for $B_{2}^{\sigma}\left(\mathbb{B}_{d}\right)$ if and only if  $Z$ satisfies the weak separation separation condition $\inf _{i \neq j} \beta\left(z_{i}, z_{j}\right)>0$ and $\mu_{Z}$ is a $B_{2}^{\sigma}\left(\mathbb{B}_{d}\right)$ Carleson measure.
\end{theo}
\begin{proof}When $0 < \sigma<1 / 2$, this theorem is given by \cite[Theorem 3]{Arcozzi}. When $\sigma=1/2$,  since $B_{2}^{1/2}\left(\mathbb{B}_{d}\right)$ has the complete Pick property, we obtain the theorem by \cite[Theorem 1.1]{john2017}.
\end{proof}
Namely, in a different fashion with respect the polydisc case, the deterministic setting for the Besov-Sobolev space has its interpolating sequences well-understood and characterized by weak separation and a Carleson measure condition. Therefore, in order to find the 0-1 Kolmogorov law for interpolating sequences for $\mathbb{B}_2^\sigma$, it suffices to find the cut-off conditions on the detrministic radii for the associated sequence with randomly chosen arguments to be weakly separated and to generate a Carleson measure almost surely. This is the intent of the second part of our work. Random sequences in the unit ball are constructed as follows. Let $\Lambda(\omega)=\left\{\lambda_{j}\right\}$ with $\lambda_{j}=\rho_{j} \xi_j(\omega)$ where $\xi_{j}(\omega)$ is a sequence of independent
random variables, all uniformly distributed on the unit sphere and $\rho_{j} \in[0,1)$ is a sequence of a priori fixed radii.
There is an interesting thing about the random interpolating sequences in the Besov-Sobolev spaces on the unit ball. As we will see, for $d\geq2$ a random sequence $\{\lambda_n\}$ is an interpolating sequence almost surely if and only if $\sum_{n}(1-|\lambda_n|)\delta_{\lambda_n}$ is a Carleson measure on $B_{2}^{\sigma}\left(\mathbb{B}_{d}\right)$ almost surely. Moreover, the characterization for almost surely interpolating sequences is strictly stronger that the characterization for almost surely weakly separated sequences.

For any $m\in \N$, let
$$N_m= \# \{\lambda_j\in \Lambda(\omega): m\frac{\ln 2}{2} \leq \beta(0,\lambda_j)<(m+1)\frac{\ln 2}{2}  \},$$
where $\beta$ is the Bergman metric on the unit ball $\mathbb{B}_{d}$ in $\mathbb{C}^{d}$.
Let
$$\mathcal{I}(B_{2}^{\sigma}\left(\mathbb{B}_{d}\right)):=\{\omega:\Lambda(\omega) \text{ is an interpolating sequence for $B_{2}^{\sigma}\left(\mathbb{B}_{d}\right)$}\}.$$
The following result is obtained regarding a 0-1 Komolgorov law for interpolating sequences on the unit ball. We only work on the case of $0 < \sigma\leq 1/2$ and $d\geq2$.
When $\sigma=d/2$, it is well-known that $B_{2}^{d/2}\left(\mathbb{B}_{d}\right)$ is the Hardy space. By \cite[Theorem 3.3]{Massaneda}, we know  that
$$\p\{\mathcal{I}(B_{2}^{d/2}\left(\mathbb{B}_{d}\right))\}=1 \text{ if and only if } \sum_{m=0}^{\infty}2^{-m}N^2_m<\infty.$$
When $d=1$ and $0<\sigma\leq1/4$, by $(i)$ in \cite[Theorem 1.5]{wick2019} we know that
$$\p\{\mathcal{I}(B_{2}^{\sigma}\left(\mathbb{D}\right))\}=1 \text{ if and only if } \sum_{m=0}^{\infty}2^{-2\sigma m}N_m<\infty.$$
When $d=1$ and $1/4<\sigma<1/2$, by $(ii)$ in \cite[Theorem 1.5]{wick2019} we know that
$$\p\{\mathcal{I}(B_{2}^{\sigma}\left(\mathbb{D}\right))\}=1 \text{ if and only if } \sum_{m=0}^{\infty}2^{-m}N^2_m<\infty.$$
In our case, we have the follows.
\begin{theo}\label{interpolation:besove}
Let $0 < \sigma\leq 1/2$ and $d\geq2$. Then
\begin{itemize}
\item[(i)] If $$\sum_{m=0}^{\infty}2^{-2\sigma m}N_m<\infty,$$ then $\p\{\mathcal{I}(B_{2}^{\sigma}\left(\mathbb{B}_{d}\right))\}=1$;
\item[(ii)] If $$\sum_{m=0}^{\infty}2^{-2\sigma m}N_m=\infty,$$ then $\p\{\mathcal{I}(B_{2}^{\sigma}\left(\mathbb{B}_{d}\right))\}=0$.
\end{itemize}
\end{theo}

Section \ref{sec:pre} will construct the necessary technical tools for the proof of our main results. Section \ref{sec:polydisc} provides the proof of Theorem \ref{theo:polydisc} and Theorem \ref{theo:intpolyh2}, and characterize random interpolating sequences for $\h^\infty(\D^d)$ for some specific choice of the radii in $(r_n)_{n\in\N}$. Finally, Section \ref{sec:ball} proves Theorem \ref{interpolation:besove} and studies the uniform separation on the unit ball.\\

We would like to thank Nikolaos Chalmoukis for some useful comments that led to the final version of Theorem \ref{theo:polydisc}. We would also like to thank the referees for their valuable suggestions.

\section{Preliminary Results}
\label{sec:pre}
This section contains relatively general results that are going to be used throughout the proof of Theorem \ref{theo:polydisc}.  Deterministic and probabilistic tools will be separately analyzed.

\subsection{Deterministic tools}

Double sums are extensively used throughout this work. In particular the fact that, for a certain class of double sums involving exponential decay, the terms of the sums on their diagonals contain all the necessary informations to bound the whole sums:

\begin{lemma}
\label{lemma:strong:ineq}
Let $s\geq1$,  and let $(A_m)_{m\in\N}$ and $(B_k)_{k\in\N}$ be two sequences of positive numbers. Then there exists some constant $C=C_s>0$ such that
\[
\sum_{m, k\in\N}\frac{A_mB^\frac{1}{s}_k}{(2^m+2^k)^\frac{1}{s}}\leq C_s\left(\max\left\{\sum_{m\in\N}A_m^{1+\frac{1}{s}}2^{-\frac{m}{s}}, \sum_{k\in\N}B_k^{1+\frac{1}{s}}2^{-\frac{k}{s}}\right\}+\sum_{m\in\N}A_mB_m^\frac{1}{s}2^{-\frac{m}{s}}\right).
\]
\end{lemma}
\begin{proof}
First observe that
\[
\begin{split}
&\sum_{m, k\in\N}\frac{A_mB^\frac{1}{s}_k}{(2^m+2^k)^\frac{1}{s}}\\
\lesssim&\sum_{k> m}\frac{A_mB^\frac{1}{s}_k}{(2^m+2^k)^\frac{1}{s}}+\sum_{ k<m}\frac{A_mB^\frac{1}{s}_k}{(2^m+2^k)^\frac{1}{s}}+\sum_{m\in\N}A_mB_m^\frac{1}{s}2^{-\frac{m}{s}}.
\end{split}
\]
Let's first estimate the sum in $k>m$:
\[
\begin{split}
&\sum_{k> m}\frac{A_mB^\frac{1}{s}_k}{(2^m+2^k)^\frac{1}{s}}\leq C_s\sum_{m=1}^\infty A_m2^{-\frac{m}{s}}\sum_{k=1}^\infty B_{m+k}^\frac{1}{s}2^{-\frac{k}{s}}\\
=C_s&\sum_{k=1}^\infty2^{-\frac{k}{s+1}}\sum_{m=1}^\infty A_m2^{-\frac{m}{s+1}}B_{m+k}^\frac{1}{s}2^{-\frac{m+k}{s(s+1)}}\\
\leq C_s&\sum_{k=1}^\infty2^{-\frac{k}{s+1}}\left(\sum_{m=1}^\infty A_m^{1+\frac{1}{s}}2^{-\frac{m}{s}}\right)^\frac{s}{s+1}\left(\sum_{m=1}^\infty B_{m+k}^{1+\frac{1}{s}}2^{-\frac{m+k}{s}}\right)^{\frac{1}{s+1}}\\
\leq&C_s~\max\left\{\sum_{m\in\N}A_m^{1+\frac{1}{s}}2^{-\frac{m}{s}}, \sum_{k\in\N}B_k^{1+\frac{1}{s}}2^{-\frac{k}{s}}\right\},
\end{split}
\]
thanks to Holder's inequality with dual exponents $1+1/s$ and $s+1$. The sum in $m>k$ is estimated analogously. This concludes the proof.
\end{proof}
Our take away from Lemma \ref{lemma:strong:ineq} is the following
\begin{coro}
\label{coro:strong}
Let $s\geq1$, $d\geq1$ and $(N_m)_{m\in\N^d}$ be a sequence of positive numbers so that
\begin{equation}
\label{eqn:hypsumd}
\sum_{m\in\N^d}N_m^{1+\frac{1}{s}}2^{-\frac{|m|}{s}}<\infty.
\end{equation}
 Then
\begin{equation}
\label{eqn:ineq:double}
\sum_{m\in\N^d}N_m\sum_{k\in\N^d}N_k^\frac{1}{s}\left(\prod_{i=1}^d\frac{1}{2^{m_i}+2^{k_i}}\right)^\frac{1}{s}<\infty.
\end{equation}
\end{coro}
\begin{proof}
The proof is by induction on $d$:
\begin{description}
\item[$d=1$] apply Lemma \ref{lemma:strong:ineq} to $A_m=B_m=N_m$;
\item[$d\geq2$] suppose that \eqref{eqn:ineq:double} is true for $d-1$, and let $(N_m)_{n\in\N^d}$ be a sequence of positive numbers. Then, by applying Lemma \ref{lemma:strong:ineq},
\[
\begin{split}
&\sum_{m\in\N^d}N_m\sum_{k\in\N^d}N_k^\frac{1}{s}\left(\prod_{i=1}^d\frac{1}{2^{m_i}+2^{k_i}}\right)^\frac{1}{s}\\
=&\sum_{\tilde{m}, \tilde{k} \in\N^{d-1}}\prod_{i=1}^{d-1}\left(\frac{1}{2^{\tilde{k}_i}+2^{\tilde{m}_i}}\right)^\frac{1}{s}\sum_{m_1, k_1\in\N}\frac{N_{(m_1, \tilde{m})}N_{(k_1, \tilde{k})}^\frac{1}{s}}{(2^{m_1}+2^{k_1})^\frac{1}{s}}\\
\underset{\sim}{<}&\sum_{\tilde{m}, \tilde{k} \in\N^{d-1}}\prod_{i=1}^{d-1}\left(\frac{1}{2^{\tilde{k}_i}+2^{\tilde{m}_i}}\right)^\frac{1}{s}\max\left\{\sum_{m_1\in\N}N_{(m_1, \tilde{m})}^{1+\frac{1}{s}}2^{-\frac{m_1}{s}}\,,\,\sum_{m_1\in\N}N_{(m_1, \tilde{k})}^{1+\frac{1}{s}}2^{-\frac{m_1}{s}}\right\}+\\
+&\sum_{\tilde{m}, \tilde{k}\in\N^{d-1}}\prod_{i=1}^{d-1}\left(\frac{1}{2^{\tilde{m}_i}+2^{\tilde{k}_i}}\right)^\frac{1}{s}\sum_{m_1\in\N}N_{(m_1, \tilde{m})}N_{(m_1, \tilde{k})}^\frac{1}{s}2^{-\frac{m_1}{s}}\\
\leq&\sum_{\tilde{m}, \tilde{k} \in\N^{d-1}}\prod_{i=1}^{d-1}\left(\frac{1}{2^{\tilde{k}_i}+2^{\tilde{m}_i}}\right)^\frac{1}{s}\sum_{m_1\in\N}N_{(m_1, \tilde{m})}^{1+\frac{1}{s}}2^{-\frac{m_1}{s}}+\\
+&\sum_{\tilde{m}, \tilde{k} \in\N^{d-1}}\prod_{i=1}^{d-1}\left(\frac{1}{2^{\tilde{k}_i}+2^{\tilde{m}_i}}\right)^\frac{1}{s}\sum_{m_1\in\N}N_{(m_1, \tilde{k})}^{1+\frac{1}{s}}2^{-\frac{m_1}{s}}+\\
+&\sum_{\tilde{m}, \tilde{k}\in\N^{d-1}}\prod_{i=1}^{d-1}\left(\frac{1}{2^{\tilde{m}_i}+2^{\tilde{k}_i}}\right)^\frac{1}{s}\sum_{m_1\in\N}N_{(m_1, \tilde{m})}N_{(m_1, \tilde{k})}^\frac{1}{s}2^{-\frac{m_1}{s}}\\
=:&I_1+I_2+I_3,
\end{split}
\]
where the index $m$ in $\N^d$ is written as $(m_1, \tilde{m})$, with $m_1$ in $\N$ and $\tilde{m}$ in $\N^{d-1}$. Observe that, thanks to \eqref{eqn:hypsumd}, $I_1$ and $I_2$ converge. As for $I_3$, we can change the order of summation and apply the case $d-1$. Which yields
\[
\begin{split}
&\sum_{m\in\N^d}N_m\sum_{k\in\N^d}N_k^\frac{1}{s}\left(\prod_{i=1}^d\frac{1}{2^{m_i}+2^{k_i}}\right)^\frac{1}{s}\\
\underset{\sim}{<}&\sum_{m_1\in\N}2^{-\frac{m_1}{s}}\sum_{\tilde{m}, \tilde{k}\in\N^{d-1}}N_{(m_1, \tilde{m})}N_{(m_1, \tilde{k})}^\frac{1}{s}\prod_{i=1}^{d-1}\left(\frac{1}{2^{\tilde{m}_i}+2^{\tilde{k}_i}}\right)^\frac{1}{s}\\
\underset{\sim}{<}&\sum_{m_1\in\N}2^{-\frac{m_1}{s}}\sum_{\tilde{m}\in\N^{d-1}}N_{(m_1, \tilde{m})}^{1+\frac{1}{s}}2^{-\frac{|\tilde{m}|}{s}}\\
=&\sum_{m\in\N^d}N_m^{1+\frac{1}{s}}2^{-\frac{|m|}{s}}<\infty.
\end{split}
\]
\end{description}
\end{proof}

\subsection{Random tools}
Fairly elementary facts from probability theory are exploited in the proofs. All the events and the random variables that are considered will be defined on the same probability space $(\Omega, \A, \p)$. For a comprehensive treatment of the probabilistic results used, see \cite{Billingsley}.

The first tool is the Borel-Cantelli Lemma.  Recall that, given a sequence $(A_n)_{n\in\N}$ of events in $\A$, then
\[
\limsup_{n\in\N}A_n:=\bigcap_{k\in\N}\bigcup_{n\geq k }A_n
\]
denotes the event made of those $\omega$ in $\Omega$ that belong to infinitely many of the events in $(A_n)_{n\in\N}$.
\begin{theo}[Borel-Cantelli Lemma]
\label{theo:bc}
Let $(A_n)_{n\in\N}$ be a sequence of events in $\A$. Then
\begin{itemize}
\item[(i)] If $\displaystyle{\sum_{n\in\N}}\p(A_n)<\infty$, then $\displaystyle\p\left(\limsup_{n\in\N}A_n\right)=0$;
\item[(ii)] If $\displaystyle{\sum_{n\in\N}}\p(A_n)=\infty$ and the events in $(A_n)_{n\in\N}$ are independent, then $\displaystyle\p\left(\limsup_{n\in\N}A_n\right)=1$.
\end{itemize}
\end{theo}
 Given a random variable $X$ on $\Omega$, its mean value (or expectation) will be denoted by
\[
\E(X):=\int_{\Omega}X\,d\p.
\]
In particular, if $\E(X)<\infty$, then $\p\{X=\infty\}=0$. \\
Another classic tool from probability that will be used is Jensen's Inequality:
\begin{theo}[Jensen's Inequality]
Let $X$ be a real-valued random variable on $\Omega$, and let $\phi\colon\R\to\R$ be a convex function. Then
\[
\E(\phi(X))\geq\phi(\E(X)).
\]
\end{theo}
In particular, since
\[
t\in(0, \infty)\mapsto t^{\frac{1}{s}}
\]
is concave, for any $s\geq1$, this gives
\begin{equation}
\label{eqn:jensen}
\E\left(X^\frac{1}{s}\right)\leq\E(X)^\frac{1}{s},
\end{equation}
for any positive random variable $X$ on $\Omega$, by applying Jensen's inequality to $\phi(t)=-t^\frac{1}{s}$. \\
We can now prove Lemma \ref{lemma:finitesumas}, a tool for the proofs of Theorem \ref{theo:polydisc}:
\begin{lemma}
\label{lemma:finitesumas}
Let $(X^i_{n, j})_{n, j\in\N}$ be a sequence of positive random variables, for any $i=1,\dots,d$. Set
\[
m(n, j):=\min_{i=1,\dots,d}X^i_{n, j},\qquad p(n, j)=\prod_{i=1}^dX^i_{n, j}.
\]
Assume that
\[
\sum_{j\in\N}\left(\sum_{k\in\N}\E(p(k, j))\right)^\frac{1}{d}<\infty.
\]
Then
\[
\sup_{n\in\N}\sum_{j\ne n} m(n, j)
\]
is bounded almost surely.
\end{lemma}
\begin{proof}
Since, for any $n\ne j$  in $\N$, $$m(n, j)\leq p(n, j)^\frac{1}{d}\leq\left(\sum_{k\ne j}p(k, j)\right)^\frac{1}{d},$$ we have
\[
\sup_{n\in\N}\sum_{j\ne n}m(n, j)\leq\sum_{j\in\N}\left(\sum_{k\ne j}p(k, j)\right)^\frac{1}{d}.
\]
Thus
\[
\E\left(\sup_{n\in\N}\sum_{j\ne n}m(n, j)\right)\leq\sum_{j\in\N}\E\left(\left(\sum_{k\in\N}p(k, j)\right)^\frac{1}{d}\right)\leq\sum_{j\in\N}\left(\sum_{k\in\N}\E(p(k, j))\right)^\frac{1}{d}.
\]
\end{proof}
\section{Random Sequences in the Polydisc}
\label{sec:polydisc}
This section is devoted to the proof of Theorem \ref{theo:polydisc} and Theorem \ref{theo:intpolyh2}. The events $\mathcal{U}(\D^d)$, $\mathcal{W}(\D^d)$, $\mathcal{C}(\h^2(\D^d))$ and $\tilde{\mathcal{I}}(\D^d)$ will be analyzed separately.

\subsection{Weak Separation}
For weak separation in the polydisc, it turns out that Cochran's argument in \cite[Th. 2]{coch} extends to the higher dimensional case:
\begin{proof}[Proof of Theorem \ref{theo:polydisc}, (i)]
For the sake of readability, we will adapt Cochran's proof only to the case $d=2$: the proof will lift appropriately to any $d>1$.
Assume first that $\sum_{m\in\N^2}N_m^22^{-|m|}=\infty$ and let $l$ be in $\N$. Define
$$
A_l:=\bigcup_{r\ne n}\{\rho_G(\lambda_r, \lambda_n)\leq5\cdot2^{-l} \}
$$
as the set of those $\omega$ in $\Omega$ such that there exists a pair of distinct indices $n$ and $r$ so that the Gleason distance between $\lambda_n(\omega)$ and $\lambda_r(\omega)$ is controlled by, roughly, $2^{-l}$.
Since $\mathcal{W}(\D^d)^c\subseteq\bigcap_{l\in\N}A_l$, it suffices to show that $\p(A_l)=1$ for any $l$ in $\N$.

For any $m$ in $\N^2$, partition $I_{m}$ into $2^{2l}$ "rectangles" of the form
\[
\left\{(z^1,z^2)\in\D^2\,|\, \frac{1}{2^{m_i-1}}+\frac{r_i}{2^{m_i+l}}\leq1-|z^i|< \frac{1}{2^{m_i-1}}+\frac{r_i}{2^{m_i+l}}\right\},\qquad r_i=1, \dots, 2^l
\]
and observe that at least one of these rectangles, say $R_{m}$, must contain at least  $M_m:=N_m/2^{2l}$ points of $\Lambda$. Let
\[
B_{m}:=\bigcup_{r\ne n}\left\{\lambda_r\in R_m, \lambda_n\in R_m, |\theta^1_r-\theta^1_n|\leq\pi\cdot2^{-(m_1+l)}, |\theta^2_n-\theta^2_r|\leq\pi\cdot2^{-(m_2+l)}\right\}.
\]
Since
\[
\limsup_{m}B_{m}\subseteq A_l
\]
and the events $B_{m}$ are independent, by the Borel-Cantelli Lemma, Theorem \ref{theo:bc}, it suffices to show that $\sum_{m\in\N^2}\p(B_{m})=\infty$.

In order to estimate the probability of each $B_{m}$ from below, we give an upper bound for $\p(B_{m}^c)$. If $\tau$ is in $\T^2$, let $S_{m}(\tau)$ be a "rectangle" in $\T^2$ centered at $\tau$ with basis $2^{-(m_1+l)}$ and height $2^{-(m_2+l)}$. If $\tau_n=(e^{i\theta^1_n}, e^{i\theta^2_n})$, then thanks to the independence of $(\tau_n)_{n\in\N}$ we have
\[
\begin{split}
\p(B_{m}^c)&\leq\p\left(\left\{\tau_1\in\T^2, \tau_2\in\T^2\setminus S_{m}(\tau_1), \dots, \tau_{M_{m}}\in S\setminus\bigcup_{j=1}^{M_{m}-1}S_{m}(\tau_j)\right\}\right)\\
&\leq\left(1-2^{-(|m|+2l)}\right)\left(1-\frac{3}{2}\cdot2^{-(|m|+2l)}\right)\dots\left(1-\frac{M_{m}}{2}\cdot2^{-(|m|+2l)}\right)\\
&=\prod_{j=2}^{M_{m}}\left(1-j\cdot2^{-(|m|+2l+1)}\right).
\end{split}
\]
If $\liminf_{m}\p(B_m^c)<1$,  then $\p(B_{m})$ is uniformly bounded away from $0$ infinitely many times, and $\sum_{m\in\N^2}\p(B_{m})=\infty$ trivially.

On the other hand, if $\lim_{|m|\to\infty}\p(B^c_{m})=1$, then
\[
\begin{split}
\p(B_{m})&\geq1-\prod_{j=2}^{M_{m}}\left(1-j\cdot2^{-(|m|+2l+1)}\right)\\
&\underset{|m|\to\infty}{\sim}-\log\prod_{j=2}^{M_{m}}\left(1-j\cdot2^{-(|m|+2l+1)}\right)\\
&=-\sum_{j=2}^{M_{m}}\log\left(1-j\cdot2^{-(|m|+2l+1)}\right)\\
&\geq\sum_{j=2}^{M_{m}}j\cdot2^{-(|m|+2l+1)}\\
&\underset{|m|\to\infty}{\sim}\frac{M_m^22^{-|m|}}{2^{2l+2}}\geq\frac{N_{m}^22^{-|m|}}{2^{6l+2}},
\end{split}
\]
which is the general term of a divergent series\\
 To conclude the proof of Theorem \ref{theo:polydisc}, part (i), it suffices to show that a random sequence $\Lambda$ in $\D^d$ is almost surely weakly separated whenever \eqref{eqn:sum:polydisc} holds. To do so, let
\[
\Omega_m:=\bigcup_{r\ne n}\left\{\lambda_r\in I_m, \lambda_n\in I_m, |\theta^1_r-\theta^1_n|\leq\pi\cdot2^{-m_1}, |\theta^2_n-\theta^2_r|\leq\pi\cdot2^{-m_2}\right\}.
\]
Then
\[
\p(\Omega_m)\leq\binom{N_m}{2}2^{-|m|}\leq\frac{1}{2}N_{m}^22^{-|m|},
\]
and the Borel-Cantelli Lemma provides that, almost surely, any pair $(\lambda_n, \lambda_r)$ in all but finitely many "rectangles" $I_m$ satisfies
\begin{equation}
\label{eqn:wsangles}
|\theta_n^1-\theta_r^1|>\pi2^{-m_1}\quad\text{or}\quad|\theta_n^2-\theta_r^2|>\pi2^{-m_2}.
\end{equation}
The same argument applies for the right-shifted "rectangles" $I'_m$ of the form
\[
\left\{1-\frac{3\cdot2^{-m_1}}{4}\leq|z_1|<1-\frac{3\cdot2^{-(m_1+1)}}{4}, 1-2^{-m_2}\leq|z_2|<1-2^{-(m_2+1)}\right\}
\]
and the up-shifted "rectangles" $I''_m$ of the form
\[
\left\{1-2^{-m_1}\leq|z_1|<1-2^{-(m_1+1)}, 1-\frac{3\cdot2^{-m_2}}{4}\leq|z_2|<1-\frac{3\cdot2^{-(m_2+1)}}{4}\right\}.
\]
This ensures that all but finitely many pairs $(\lambda_n,\lambda_r)$ in $\Lambda$ so that both
\[
|\lambda_n^1-\lambda_r^1|\simeq 2^{-m_1}
\]
and
\[
|\lambda_n^2-\lambda_r^2|\simeq 2^{-m_2}
\]
have property \eqref{eqn:wsangles}. Therefore, see \cite[Claim, p. 741]{coch} $\Lambda$ is almost surely weakly separated.
\end{proof}

\subsection{Uniform Separation}
While weak separation behaves essentially in the same way as the dimension $d$ grows, the sufficient condition in \eqref{eqn:polydisc:strong} for almost sure uniform separation picks up a dependence on $d$. As it will be shown, this is due to some estimates on the expected value of quantities related to the (random) Gleason distances between the points in $\Lambda$.

It will also be explained how \eqref{eqn:polydisc:strong} can be improved for some choices of $(r_n)_{n\in\N}$. As a corollary, a cut off condition for $\Lambda$ to be almost surely $\h^\infty(\D^d)$-interpolating for some types of random sequences in the polydisc will be given.\\

Let $s_d$ be the Szegö kernel on $\D^d$. Then the Hardy space $\h^2(\D^d)$ is the reproducing kernel Hilbert space $\mathcal{H}_{s_d}$. Denote the normalized Szegö kernel by
\[
S_d(z, w):=\prod_{i=1}^d\frac{\sqrt{(1-|z^i|^2)(1-|w^i|^2)}}{1-z^i\overline{w^i}},
\]
and observe that, for any $z$ and $w$ in $\D^d$,
\begin{equation}
\label{eqn:strong:szego}
\rho_G(z, w)^2=1-\min_{i=1,\dots, d}|S_1(z^i, w^i)|^2.
\end{equation}
Given a random sequence $\Lambda$ in $\D^d$ denote, for the sake of readability,
\[
S^i(n, j):=S_1(\lambda_n^i, \lambda_j^i)
\]
and
\[
S_d(n, j):=S_d(\lambda_n, \lambda_j).
\]
Thanks to \eqref{eqn:strong:szego}, uniform separation can be achieved from weak separation and a uniform bound on sums depending on the random sequences $(S^ i(n, j))_{n, j\in\N}$:
\begin{equation}
\label{eqn:strongweak}
\mathcal{U}(\D^d)=\mathcal{W}(\D^d)\cap\left\{\sup_{n\in\N}\sum_{j\ne n}\min_{i=1,\dots, d}|S^i(n, j)|^2<\infty\right\}.
\end{equation}
Observe that each $(S^i(n, j))_{n, j\in\N}$ is a sequence of random variables on $\Omega$ which is determined, together with $\Lambda$, by $(r_{n})_{n\in\N}$. It is not surprising then that the expectation of $|S^i(n, j)|^2$ depends, for any $i$, $n$ and $j$, only on $r^i_n$ and $r_j^i$:
\begin{lemma}
\label{lemma:exp}
Let $\Lambda$ be a random sequence in $\D^d$. Then, for any $n\ne j$ in $\N$ and for any $i=1,\dots, d$,
\[
\E(|S^i(n, j)|^2)=\frac{\bigg(1-(r^i_n)^2\bigg)\bigg(1-(r^i_j)^2\bigg)}{1-\left(r^i_nr^i_j\right)^2}.
\]
\end{lemma}
\begin{proof}
Observe that\footnote{The reader should not confuse the index $i=0,\dots,d$ and $i=\sqrt{-1}$!}
\[
\begin{split}
|S^i(n, j)|^2=&\bigg(1-(r^i_n)^2\bigg)\bigg(1-(r^i_j)^2\bigg)\left|\sum_{k=0}^\infty(r^i_nr^i_j)^ke^{-ik(\theta^i_n-\theta^i_j)}\right|^2\\
=&\bigg(1-(r^i_n)^2\bigg)\bigg(1-(r^i_j)^2\bigg)\sum_{k=0}^\infty(r^i_nr^i_j)^k\sum_{l=0}^ke^{i(2l-k)(\theta^i_n-\theta^i_j)}.
\end{split}
\]
Therefore, by making use of the independence of $\theta^i_n$ and $\theta^i_j$,
\[
\begin{split}
\E(|S^i(n, j)|^2)
=&\bigg(1-(r^i_n)^2\bigg)\bigg(1-(r^i_j)^2\bigg)\sum_{k=0}^\infty(r^i_nr^i_j)^k\sum_{l=0}^k\E\left(e^{i(2l-k)\theta^i_n}\right)\E\left(e^{i(k-2l)\theta^i_2}\right)\\
=&\bigg(1-(r^i_n)^2\bigg)\bigg(1-(r^i_j)^ 2\bigg)\sum_{k=0}^\infty(r^i_nr^i_j)^{2k}\\
=&\frac{\bigg(1-(r^ i_n)^2\bigg)\bigg(1-(r^i_j)^ 2\bigg)}{1-(r^ i_nr^ i_j)^2}.
\end{split}
\]
\end{proof}

\begin{rem}
\label{rem:strong}
Let $m$ and $k$ be two multi-indices in $\N^d$, and suppose that $\lambda_n$ and $\lambda_j$ belong to $I_m$ and $I_k$, respectively. Then, thanks to Lemma \ref{lemma:exp} and \eqref{eqn:rectangles},
\[
\E(|S^i(n, j)|^2)\simeq\frac{2^{-(m_i+k_i)}}{2^{-m_i}+2^{-k_i}-2^{-(m_i+k_i)}}=\frac{1}{2^{k_i}+2^{m_i}-1}\simeq\frac{1}{2^{k_i}+2^{m_i}}.
\]
In particular, since $S^i(n, j)$ and $S^r(n, j)$ are independent for any $i\ne r$, we have
\[
\E(|S_d(n, j)|^2)\simeq\prod_{i=1}^d\frac{1}{2^{k_i}+2^{m_i}}.
\]
\end{rem}
Part (ii) of Theorem \ref{theo:polydisc} can now be proved:
\begin{proof}[Proof of Theorem \ref{theo:polydisc}, \textnormal{(ii)}] Observe that
\[
\sum_{m\in\N^d}N_m^22^{-|m|}\leq\sum_{m\in\N^d}N_m^{1+\frac{1}{d}}2^{-\frac{|m|}{d}},
\]
whenever $N_m\leq2^{|m|}$, and so under our assumption $\Lambda$ is weakly separated, thanks to Theorem \ref{theo:polydisc}, part (i). Therefore, thanks to \eqref{eqn:strongweak}, it suffices to show that the random sequence $(S_n)_{n\in\N}$ given by
\[
S_n:=\sum_{j\ne n}\min_{i=1,\dots,d}|S^i(n, j)|^2
\]
is bounded almost surely. Thanks to Lemma \ref{lemma:finitesumas}, it is enough to show that
\begin{equation}
\label{eqn:sumexp}
\sum_{j\in\N}\left(\sum_{n\in\N}\E\left(|S_d(n, j)|^2\right)\right)^\frac{1}{d}<\infty.
\end{equation}
By regrouping the terms of the double sum in \eqref{eqn:sumexp} with respect the partition $(I_m)_{m\in\N^d}$ of $\D^d$ and thanks to Remark \ref{rem:strong} and \eqref{eqn:jensen} we get
\[
\begin{split}
&\sum_{j\in\N}\left(\sum_{n\in\N}\E\left(|S_d(n, j)|^2\right)\right)^\frac{1}{d}\\
=&\sum_{m\in\N^d}\sum_{\lambda_n\in I_m}\left(\sum_{k\in\N^d}\sum_{\lambda_j\in I_k}\E(|S_d(n, j)|^2)\right)^\frac{1}{d}\\
\simeq&\sum_{m\in\N^d}N_m\left(\sum_{k\in\N^d}N_k\prod_{i=1}^d\frac{1}{2^{m_i}+2^{k_i}}\right)^\frac{1}{d}\\
\leq&\sum_{m\in\N^d}N_m\sum_{k\in\N^d}N_k^\frac{1}{d}\prod_{i=1}^d\left(\frac{1}{2^{m_i}+2^{k_i}}\right)^\frac{1}{d}.
\end{split}
\]
Corollary \ref{coro:strong}, $d=s$, concludes the proof.
\end{proof}
Condition \eqref{eqn:polydisc:strong} is not sharp. Indeed,  for some choices of $(r_n)_{n\in\N}$, we can show that the $0-1$ Kolmogorov law for $\h^\infty(\D^d)$-interpolating sequences coincide with the one for weak separation:
\begin{prop}
\label{prop:example}
Let $d=2$ and $(t_n)_{n\in\N}$ be a sequence in $(0, 1)$, and consider its Cartesian product with itself
\[
r_{n}:=(t_{n_1}, t_{n_2})\qquad n=(n_1, n_2)\in\N^2.
\]
Then the random sequence $\Lambda$ associated with $(r_n)_{n\in\N^2}$ is interpolating for $\h^\infty(\D^d)$ almost surely if an only if \eqref{eqn:sum:polydisc} holds.
\end{prop}
\begin{proof}
If $\sum_{m\in\N^2}N_m^22^{-|m|}=\infty$, then $\Lambda$ is not weakly separated almost surely, and in particular it is almost surely not interpolating. Thus it suffices to show that $\Lambda$ is $\h^\infty(\D^d)$-interpolating provided that  $\sum_{m\in\N^2}N_m^22^{-|m|}<\infty$, which, by construction of $(r_n)_{n\in\N^2}$, it is equivalent to
\[
\sum_{n\in\N}T_n^22^{-n}<\infty,
\]
 where $T_n:=\#\{l\in\N \quad|\quad 1-2^{-n}\leq t_l<1-2^{-(n+1)}\}$. By Rudowicz's Theorem, \cite{rudo}, the random sequence $T$ on $\D$ given by
\[
\tau_n:=t_ne^{i\theta_n}\qquad n\in\N
\]
is almost surely interpolating in $\D$, where $(\theta_n)_{n\in\N}$ is a sequence of i.i.d. random variables defined on a probability space $(\Omega, \A, \p)$ and distributed uniformly on the unit circle. In particular, $T$ has almost surely a sequence of so called \emph{P. Beurling functions}, that is, there exists an event $\Omega'$ so that  $\p(\Omega')=1$ and, for any $\omega$ in $\Omega'$, there exists a sequence of $\mathrm{H}^\infty(\D)$ functions $(F_{\omega, n})_{n\in\N}$ such that
\[
\begin{cases}
F_{\omega, n}(\tau_j(\omega))=\delta_{n, j}\\
\sup_{z\in\D}\sum_{n\in\N}|F_{\omega, n}(z)|<\infty.
\end{cases}
\]
Let us consider now the product probability space $(\tilde{\Omega}, \tilde{\A}, \tilde{\p})$, where $\tilde{\Omega}:=\Omega\times\Omega$, $\tilde{\A}$ is the product $\sigma$-algebra of $\A$ with itself, and
\[
\tilde{\p}(A\times B)=\p(A)\p(B),\qquad A, B\in\A.
\]
Then the random variables
\[
\theta_{n_1, n_2}\colon\tilde{\Omega}\to\T^2
\]
given by
\[
\theta_{n_1, n_2}(\omega_1, \omega_2):=(\theta_{n_1}(\omega_1), \theta_{n_2}(\omega_2))
\]
are uniformly distributed in $\T^2$ and independent. Thus we can think of the random sequence $\Lambda$ as
\[
\lambda_{n_1, n_2}(\omega_1, \omega_2):=(r_{n_1}e^{i\theta_{n_1}(\omega_1)}, r_{n_2}e^{i\theta_{n_2}(\omega_2)})\qquad(\omega_1, \omega_2)\in\tilde{\Omega}.
\]
Let $\Omega'':=\Omega'\times\Omega'$ and define, for any $n=(n_1, n_2)$ in $\N^2$ and $\tilde{\omega}=(\omega_1, \omega_2)$ in $\Omega''$ the $\h^\infty(\D^2)$ function
\[
G_{\tilde{\omega}, n}(z_1, z_2)=F_{\omega_1, n_1}(z_1)~F_{\omega_2, n_2}(z_2)\qquad(z_1, z_2)\in\D^2.
\]
Then $(G_{\tilde{\omega}, n})_{n\in\N^2}$ is a set of P. Beurling functions for $\Lambda(\tilde{\omega})$, and in particular $\Lambda(\tilde{\omega})$ is $\h^\infty(\D^d)$-interpolating for any $\tilde{\omega}$ in $\Omega''$. Since $\tilde{\p}(\Omega'')=\p(\Omega')^2=1$, $\Lambda$ is interpolating for $\h^\infty(\D^d)$ almost surely.
\end{proof}

The argument in Proposition \ref{prop:example} can be easily extended to any $d>1$ to show that, whenever the sequence of radii $(r_n)_{n\in\N}$ is the Cartesian product of $d$ sequences in $[0, 1)$, then \eqref{eqn:sum:polydisc} encodes all random sequences that are almost surely interpolating for $\h^\infty(\D^d)$. For a general choice of $(r_n)_{n\in\N}$ the following question remains open:
\begin{question}
\label{q:wi}
Is any random sequence $\Lambda$ in $\D^d$ satisfying \eqref{eqn:sum:polydisc} uniformly separated?  Or else, does there exist a choice of $(r_n)_{n\in\N}$ so that the random sequence $\Lambda$ obtained is almost surely weakly separated but not uniformly separated?
\end{question}
\subsection{Carleson Measures}
\label{sec:cm}
The same idea that was used for random uniform separation works for the proof of Theorem \ref{theo:polydisc}, part (iii), modulo some adaptations. Let $Z=(z_n)_{n\in\N}$ be a sequence in $\D^d$ and consider the Szegö Grammian
\[
G:=(S_d(z_n, z_j))_{n, j\in\N}
\]
associated with the sequence $Z$. Therefore,
\begin{theo}
\label{theo:carl}
The following are equivalent:
\begin{description}
\item[(i)] $\mu_Z$ is a Carleson measure for $\h^2(\D^d)$;
\item[(ii)] $G\colon l^2\to l^2$ is bounded.
\end{description}
\end{theo}
A proof of Theorem \ref{theo:carl} can be found in \cite[Th. 9.5]{john}. Moreover, a standard operator theory argument gives that any sufficiently strong decay of the coefficients of $G$ outside its diagonal implies that $G$ is bounded (above and below):
\begin{lemma}
\label{lemma:Grammian}
Let $A=(a_{n, j})_{n, j\in\N}\colon l^2\to l^2$ be invertible and self adjoint. Suppose that $a_{i,i}=1$ for any $i$ in $\N$, and that
\begin{equation}
\label{eqn:hs}
\sum_{j\in\N}\sum_{n\ne j}|a_{n, j}|^2=M^2<\infty.
\end{equation}
Then $A$ is bounded above and below.
\end{lemma}
\begin{proof}
Such an $A$ can be written as $A=Id+H$, where $H$ is a Hilbert-Schmidt operator. Let $(y_n)_{n\in\N}$ be the sequence of eigenvalues of $A$, and let $(x_n)_{n\in\N}$ be the eigenvalues of $H$. Since $H$ is a Hilbert-Schmidt operator, then
\[
\sum_{n\in\N}|x_n|^2<\infty,
\]
and since $A=Id+H$ we have that $y_n=1+x_n$ for any $n$. Since $A$ is invertible, none of the $y_n$ is null.  Moreover, being a self-adjoint infinite matrix, $A$ is bounded by $\sup_{n\in\N}|y_n|$ and bounded below by $\inf_{n\in\N}|y_n|$. Since $x_n$ converges to $0$, the two quantities are bounded above and below, hence the result.
\end{proof}
\begin{rem}
In the above proof one uses only the fact that $x_n$ goes to $0$, as $n\to\infty$. Therefore the same conclusion holds if we assume $H$ to be compact.
\end{rem}
Let $\Lambda$ be a random sequence in $\D^d$. Thanks to Lemma \ref{lemma:Grammian}, to show that $\p(\mathcal{C}(\h^2(\D^d)))=1$ it is enough to show that the random Grammian associated to $\Lambda$ has a strong decay outside its diagonal almost surely:
\begin{proof}[Proof of Theorem \ref{theo:polydisc},(iii)]
It suffices to show that
\begin{equation}
\label{eqn:carl:sum}
\sum_{j\in\N}\sum_{n\ne j}\E(|S_d(n, j)|^2)<\infty.
\end{equation}
Indeed, if \eqref{eqn:carl:sum} holds, then
\[
\sum_{j\in\N}\sum_{n\ne j}|S_d(n, j)|^2<\infty
\]
almost surely, and Lemma \ref{lemma:Grammian} would conclude the proof.  By Remark \ref{rem:strong} and by regrouping the sum in \eqref{eqn:carl:sum} with respect the partition $(I_m)_{m\in\N^d}$ of $\D^d$, one obtains
\[
\sum_{j\in\N}\sum_{n\ne j}\E(|S_d(n, j)|^2)\leq C~\sum_{m, k\in\N^d}N_mN_k\left(\prod_{i=1}^d\frac{1}{2^{m_i}+2^{k_i}}\right).
\]
Corollary \ref{coro:strong}, $s=1$, concludes the proof.
\end{proof}
\subsection{Almost Orthogonal Random Grammians}
\label{sec:hs}
Equation \eqref{eqn:hs} is a rather strong condition for an infinite matrix $A$. Indeed, other than implying that $A$ is bounded, it says that $A-Id$ is a \emph{Hilbert-Schmidt operator} on $l^2$, i.e., that for any choice of an orthonormal basis $(e_n)_{n\in\N}$ of $l^2$
\[
\sum_{n\in\N}||(A-Id)e_n||^2<\infty.
\]
If $A=G$ is a Szegö Grammian associated to a sequence $Z=(z_n)_{n\in\N}$ in the polydisc, it comes natural to ask whether such an almost orthogonality condition on the kernels at the points of $Z$ translates to interpolation properties on the points of the sequences:
\begin{question}
\label{q:hs}
Let $d\geq2$. Is a sequence $Z$ in $\D^d$ intepolating for $\h^\infty(\D^d)$, provided that its Szegö Grammian can be written as $G=Id+H$, where $H$ is a Hilbert-Schmidt operator on $l^2$?
\end{question}
The case $d=1$ of Question \ref{q:hs} has a positive answer. For any  sequence $Z$ in the unit disc, let
\[
\delta_n:=\prod_{j\ne n}\rho(z_n, z_j)
\]
be the hyperbolic distance from $z_n$ to the rest of the sequence. By Carleson interpolation Theorem, $Z$ is interpolating if and only if $\inf_{n\in\N}\delta_n>0$. On the other hand, \cite{thin}, $G-Id$ is a Hilbert-Schmidt operator if and only if
\[
\sum_{n\in\N}1-\delta_n<\infty,
\]
giving that $Z$ is interpolating rather comfortably.\\
Another motivation to answer Question \ref{q:hs} comes from random interpolating sequences for $\h^\infty(\D^d)$. We proved in Section \ref{sec:cm} that the random Grammian associated to a random sequence $\Lambda$ in the polydisc differs from the identity by a Hilbert-Schmidt operator, provided that the sum in $\eqref{eqn:sum:polydisc}$. Conversely, if $Z$ is not weakly separated, then infinitely many entries outside the diagonal of its Szegö Grammian are arbitrarily close to $1$ in absolute value, hence $G-Id$ is not Hilbert-Schmidt. Namely,
\begin{equation}
\label{eqn:01law}
\p(G-Id\quad\text{is Hilbert-Schmidt})=\begin{cases}
1\quad\text{if}\,&\sum_{m\in\N^d}N_m^22^{-|m|}<\infty\\
0\quad\text{if}\,&\sum_{m\in\N^d}N_m^22^{-|m|}=\infty.
\end{cases}
\end{equation}
In particular, a positive answer to Question \ref{q:hs} would imply that the event $\mathcal{I}(\D^d)$ follows the same $0-1$ law of \eqref{eqn:01law}, giving the $0-1$ las for random $\h^\infty(\D^d)$-interpolating sequences.\\
Moreover, \eqref{eqn:01law} helps understanding interpolating sequences for $\h^2(\D^d)$, and it implies Theorem \ref{theo:intpolyh2}. Indeed, any invertible Szegö Grammian $(S_d(z_n, z_j))_{n, j\in\N}$ that can be written as $G=Id+H$, where $H$ is Hilbert-Schimdt, is bounded above and below, thanks to Lemma \ref{lemma:Grammian}, which in turn is equivalent to $(z_n)_{n\in\N}$ being interpolating for $\h^2(\D^d)$. On the other hand, as pointed out above if $Z$ is not weakly separated then infinitely many pairs of normalized Szegö kernels at the points of $Z$ are at an angle arbitrarily close to $0$, and hence $G$ is not bounded below. Thus
\[
\p(\tilde{\mathcal{I}}(\D^d))=\begin{cases}
1\quad\text{if}\,&\sum_{m\in\N^d}N_m^22^{-|m|}<\infty\\
0\quad\text{if}\,&\sum_{m\in\N^d}N_m^22^{-|m|}=\infty
\end{cases}.
\]
\section{Random Separation in the  Unit Ball}
\label{sec:ball}
This section is devoted to the proof of Theorem \ref{interpolation:besove}. On the other hand, we will study the uniform separation on the unit ball.
Compared with the polydisc, we use more heavily the spherical geometry of the unit ball rather than the Euclidean geometry of the Hardy spaces involved. So, the techniques used in this section are different from the previous sections.

Recall that $\Lambda(\omega)=\left\{\lambda_{j}\right\}$ with $\lambda_{j}=\rho_{j} \xi_j(\omega)$ where $\xi_{j}(\omega)$ is a sequence of independent
random variables, all uniformly distributed on the unit sphere and $\rho_{j} \in[0,1)$ is a sequence of a priori fixed radii. Depending on the distribution conditions on $\{\rho_{j}\}$ as will be discussed below, the probability that $\Lambda(\omega)$ is interpolating for Besov-Sobolev spaces $B_{p}^{\sigma}\left(\mathbb{B}_{d}\right)$, where $0 < \sigma\leq1 / 2$ is studied.

The Bergman tree $\mathcal{T}_{d}$ associated to the ball $\mathbb{B}_{d}$ with the structure constants $1$ and $\frac{\ln 2}{2}$ is needed in the analysis, so we present here some details. More information can be found in \cite[pg 17]{Arcozzi2006}. Let $\rho$ be the pseudo-hyperbolic distance on the unit ball, thus
$\rho(z,w)=|\varphi_{z}(w)|$ where $\varphi_{z}(w)$ is the M$\ddot{\text{o}}$bius transform.
The Bergman metric on the unit ball $\mathbb{B}_{d}$ in $\mathbb{C}^{d}$ is given by
$$ \beta(z,w)=\frac{1}{2}\log\frac{1+\rho(z,w)}{1-\rho(z,w)}. $$ 
Further, for any $r>0$, we define
$$\mathcal{U}_{r}=\partial B_{\beta}(0, r)=\left\{z \in \mathbb{B}_{d}: \beta(0, z)=r\right\}.$$
For any $N\in \mathbb{N}$, according to \cite[Lemma 2.6]{Arcozzi2006} and the fact that $\mathcal{U}_{r}$ is a compact set, there is a positive integer $J$, a set of points $\{z^N_{j}\}_{j=1}^J$ and a set of subsets $\{Q_j^N\}_{j=1}^N$ of $\mathcal{U}_{ \frac{N\ln 2}{2}}$ such that
$$\mathcal{U}_{ \frac{N\ln 2}{2}}=\bigcup_{j=1}^J Q_j^N,$$
$$Q_i^N \cap Q_j^N=\emptyset \text{ when }i\neq j,$$
$$\mathcal{U}_{\frac{N\ln 2}{2}}\cap B_{\beta}\left(z^N_{j}, 1\right) \subset Q^N_{j} \subset \mathcal{U}_{\frac{N\ln 2}{2}}\cap B_{\beta}\left(z^N_{j}, 2 \right).$$
Let
$$K_{j}^{N}=\left\{z \in \mathbb{B}_{n}: \frac{N\ln 2}{2} \leq \beta(0, z)< \frac{(N+1)\ln 2}{2}, P_{N} z \in Q_{j}^{N}\right\},$$
where $P_{N} z$ denote the radial projection of $z$ onto the sphere $\mathcal{U}_{\frac{N\ln 2}{2}}$.
Define a tree structure on the collection of sets
\[
\mathcal{T}_{d}=\left\{K_{j}^{N}\right\}_{N \geq 0, j \geq 1}
\]
by declaring that $K_{i}^{N+1}$ is a child of $K_{j}^{N},$ written $K_{i}^{N+1} \geq K_{j}^{N},$ if the projection $P_{N }\left(z_{i}^{N+1}\right)$ of $z_{i}^{N+1}$ onto the sphere $\mathcal{U}_{\frac{N\ln 2}{2}}$ lies in $Q_{j}^{N} .$
For any $K_{j}^{N}\in \mathcal{T}_{d}$, we define $d(K_{j}^{N})$ by
$$d(K_{j}^{N})=N.$$

Given a non-negative function $h$ on $\mathbb{N}$, we say $h$ is summable if 
$$\sum_{N\in \mathbb{N}}h(N)<+\infty.$$ For $\sigma>0$, a measure $\mu$ satisfies the strengthened simple condition if there is a summable function $h(\cdot)$ such that
$$
2^{2 \sigma d(\alpha)} I^{*} \mu(\alpha) \leq C h(d(\alpha)), \quad \alpha \in \mathcal{T}_{d},
$$
where
$$I^{*} \mu(\alpha)=\sum_{\alpha'\geq\alpha,\alpha'\in\mathcal{T}_{d}}\mu(\alpha').$$
The following lemma follows from \cite[Lemma 32 and Theorem 23]{Arcozzi}.
\begin{lemma}\label{carleson}
Let $\sigma>0 .$ If $\mu$ satisfies the strengthened simple condition, then $\mu$ is a $B_{2}^{\sigma}\left(\mathbb{B}_{d}\right)$-Carleson measure on $\mathbb{B}_{d}$.
\end{lemma}

The following Lemma can be found in \cite{Billingsley}.
\begin{lemma}\label{binimiallimit}
If X is a binomial random variable with parameter $p,N$, then for every $s=0,1,2,...,$
$$\lim_{\tiny{\begin{array}{c}N\rightarrow\infty\\pN\rightarrow0\end{array}}}\frac{P(X=s)}{(pN)^s}=
\lim_{\tiny{\begin{array}{c}N\rightarrow\infty\\pN\rightarrow0\end{array}}}\frac{P(X\geq s)}{(pN)^s}=\frac{1}{s!}$$
\end{lemma}
Define
$$\mathcal{C}(B_{2}^{\sigma}\left(\mathbb{B}_{d}\right)):=\{\omega: \mu_\Lambda \text{ is a Carleson measure for $B_{2}^{\sigma}\left(\mathbb{B}_{d}\right)$}\}.$$

\begin{theo}Let $\frac{d}{2}>\sigma>0 $ and $d\geq2$. Then
\begin{itemize}
\item[(i)] If  $$\sum_{m=0}^{\infty}2^{-2\sigma m}N_m<\infty,$$ then $\p\{\mathcal{C}(B_{2}^{\sigma}\left(\mathbb{B}_{d}\right))\}=1$.
\item[(ii)] If  $$\sum_{m=0}^{\infty}2^{-2\sigma m}N_m=\infty,$$ then $\p\{\mathcal{C}(B_{2}^{\sigma}\left(\mathbb{B}_{d}\right))\}=0$.
\end{itemize}
\end{theo}

\begin{proof}
First, it will be shown  that if
$$\sum_{m=0}^{\infty}2^{-2\sigma m}N_m<\infty, $$
then
$\mu_{\Lambda(\omega)}=\sum_{j=1}^{\infty}\left(1-\left|\lambda_{j}\right|^{2}\right)^{2 \sigma} \delta_{\lambda_{j}} $ is a Carleson measure almost surely.

Since $\frac{d}{2}>\sigma>0$, there is a constant $\epsilon$ such that  $d>2\sigma+\epsilon$.  Next, it will be shown that
$$\sup_{\alpha\in \mathcal{T}_{d}}2^{(\epsilon+2\sigma)d(\alpha)}\sum_{\lambda_j\in\beta\geq\alpha}(1-|\lambda_j|^2)^{2\sigma}$$
is bounded almost surely, that is to say
$$2^{2 \sigma d(\alpha)} I^{*} \mu(\alpha)=2^{2\sigma d(\alpha)}\sum_{\lambda_j\in\beta\geq\alpha}(1-|\lambda_j|^2)^{2\sigma}\lesssim2^{-\epsilon d(\alpha)},$$
which implies $\mu_{\Lambda(\omega)}$ is a Carleson measure almost surely by Lemma \ref{carleson}. For any $\alpha$, let
$$X_{m,\alpha}=\# \{\lambda_j\in \Lambda(\omega): \lambda_j\in\beta\geq\alpha,d(\beta)=m \}.$$
By \cite[Lemma 2.8]{Arcozzi2006}, it holds
$$\left|\bigcup_{\beta\geq \alpha, d(\beta)=m} \beta\right|=c_{m,\alpha}2^{-d(\alpha)d}\left|\{z, m\theta \leq \beta(0,z)<(m+1)\theta\}\right|,$$
thus $X_{m,\alpha}$ follows the binomial distribution $B(c_{m,\alpha}2^{-d(\alpha)d}, N_m)$ and $\sup_{m,\alpha}c_{m,\alpha}=c<\infty$.
Then
\begin{align*}
2^{(\epsilon+2\sigma)d(\alpha)}\sum_{\lambda_j\in\beta\geq\alpha}(1-|\lambda_j|^2)^{2\sigma}
&=2^{(\epsilon+2\sigma)d(\alpha)}\sum_{m=d(\alpha)}^{\infty} \sum_{\lambda_j\in\beta\geq\alpha,d(\beta)=m}(1-|\lambda_j|^2)^{2\sigma}\\
&\lesssim 2^{(\epsilon+2\sigma)d(\alpha)}\sum_{m=d(\alpha)}^{\infty} 2^{-2\sigma m} X_{m,\alpha}\triangleq S_{\alpha}.
\end{align*}
Choosing $\gamma\in \mathbb{N}$ such that $-2\sigma\gamma+2\sigma+2\epsilon<-d$.
Let
$$
Y_{\alpha}=2^{(\epsilon+2\sigma)d(\alpha)}\sum_{m=d(\alpha)}^{d(\alpha)(1+\gamma)-1} 2^{-2\sigma m} X_{m,\alpha} \text{ and }
R_{\alpha}=2^{(\epsilon+2\sigma)d(\alpha)}\sum_{m=d(\alpha)(1+\gamma)}^{\infty} 2^{-2\sigma m}X_{m,\alpha}.
$$
For any constant $A$, observe that
$$\p\left(\{\omega:S_{\alpha}\geq A \}\right)\leq \p\left(\left\{\omega:Y_{\alpha}\geq \frac{A}{2} \right\}\right)+\p\left(\left\{\omega:R_{\alpha}\geq \frac{A}{2} \right\}\right).$$
For any $m$ and $\alpha$, there is an open set $S_{m,\alpha}$ such that
$$\bigcup_{\beta\geq \alpha, d(\beta)=m} \beta\subset S_{m,\alpha} \subset \{z, m\theta \leq \beta(0,z)<(m+1)\theta\}$$
and
$$|S_{m,\alpha}|=c2^{-d(\alpha)d}|\{z, m\theta \leq \beta(0,z)<(m+1)\theta\}|.$$
Let
$$\tilde{X}_{m,\alpha}=\# \{\lambda_j:  \lambda_j\in\Lambda(\omega)\cap S_{m,\alpha} \},$$
then
$\tilde{X}_{m,\alpha}$ follows the binomial distribution $B(c2^{-d(\alpha)d}, N_m)$ and $X_{m,\alpha}\leq \tilde{X}_{m,\alpha}$. Let
$$\tilde{Y}_{m,\alpha}= 2^{(\epsilon+2\sigma)d(\alpha)}\sum_{m=d(\alpha)}^{d(\alpha)(1+\gamma)-1} 2^{-2\sigma m} \tilde{X}_{m,\alpha},$$
then  $\tilde{Y}_{m,\alpha}$ follows the binomial distribution
$$B\Big(c2^{-d(\alpha)d}, 2^{(\epsilon+2\sigma)d(\alpha)}\sum_{m=d(\alpha)}^{d(\alpha)(1+\gamma)-1} 2^{-2\sigma m} N_m\Big).$$
Then, by Lemma \ref{binimiallimit}, it holds
\begin{align*}
\p\{\omega:\tilde{Y}_{\alpha}\geq \frac{A}{2}\}
&\lesssim \frac{\Big[c2^{-d(\alpha)d} 2^{(\epsilon+2\sigma)d(\alpha)}
\sum_{m=d(\alpha)}^{d(\alpha)(1+\gamma)-1} 2^{-2\sigma m} N_m\Big]^{A/2}}{(A/2)!}\\
&\lesssim \frac{\Big[c2^{-d(\alpha)d} 2^{(\epsilon+2\sigma)d(\alpha)}\Big]^{A/2}}{(A/2)!}.\\
\end{align*}
Since $\epsilon+2\sigma-d<0,$ choose $A$ big enough such that $\frac{A}{2}(\epsilon+2\sigma-d)\leq-2d$.
Thus
$$\p\left(\left\{\omega:Y_{\alpha}\geq \frac{A}{2}\right\}\right)\lesssim \p\left(\left\{\omega:\tilde{Y}_{\alpha}\geq \frac{A}{2}\right\}\right)\lesssim 2^{-2d(\alpha)d}.$$
On the other hand,
\begin{align*}
\E(R_{\alpha})&=2^{(\epsilon+2\sigma)d(\alpha)}\sum_{m=d(\alpha)(1+\gamma)}^{\infty} 2^{-2\sigma m}\E(X_{m,\alpha})\\
&=2^{(\epsilon+2\sigma)d(\alpha)}\sum_{m=d(\alpha)(1+\gamma)}^{\infty} 2^{-2\sigma m}c_{m,\alpha}2^{-d(\alpha)d} N_m\\
&\leq 2^{(\epsilon+2\sigma-d)d(\alpha)}\sum_{m=d(\alpha)(1+\gamma)}^{\infty} 2^{-2\sigma m} N_m\leq C,
\end{align*}
for some constant $C$. Without loss of generality, suppose $A/4\geq C$. Then
\begin{align*}
&\p\left(\left\{\omega:R_{\alpha}\geq \frac{A}{2}\right\}\right)= \p\left(\left\{\omega:R_{\alpha}-\E(R_{\alpha})\geq \frac{A}{2}-\E(R_{\alpha})\right\}\right)\\
&\leq \p\left(\left\{\omega:|R_{\alpha}-\E(R_{\alpha})|\geq \frac{A}{4}\right\}\right)\lesssim  \operatorname{Var}(R_{\alpha})\lesssim 2^{2(\epsilon+2\sigma)d(\alpha)}\sum_{m=d(\alpha)(1+\gamma)}^{\infty} 2^{-4\sigma m}2^{-d(\alpha)d} N_m\\
&\lesssim 2^{2(\epsilon+2\sigma)d(\alpha)}2^{-d(\alpha)d}\sum_{m=d(\alpha)(1+\gamma)}^{\infty} 2^{-2\sigma m}\lesssim 2^{2(\epsilon+2\sigma)d(\alpha)}2^{-d(\alpha)d} 2^{-2\sigma d(\alpha)(1+\gamma)}\\
&\lesssim2^{-d(\alpha)d} 2^{d(\alpha)(-2\sigma\gamma+2\sigma+2\epsilon)}\lesssim2^{-d(\alpha)2d}.
\end{align*}
Thus,
\begin{align*}
\sum_{\alpha\in \mathcal{T}_{d}} \p\left(\left\{\omega:S_{\alpha}\geq A \right\}\right)
&=\sum_{k=0}^{\infty}\sum_{\alpha\in \mathcal{T}_{d},d(\alpha)=k} \p\left(\left\{\omega:S_{\alpha}\geq A \right\}\right)\\
&\lesssim\sum_{k=0}^{\infty}\sum_{\alpha\in \mathcal{T}_{d},d(\alpha)=k}2^{-d(\alpha)2d}\lesssim\sum_{k=0}^{\infty}2^{-kd}<\infty,
\end{align*}
which means that $S_{\alpha}$ is bounded almost surely. Thus $\mu_{\Lambda(\omega)}$ is a Carleson measure almost surely.

On the other hand, if $\sum_{m=0}^{\infty}2^{-2\sigma m}N_m=\infty$, then
$$\int_{\mathbb{B}_d}d\mu_{\Lambda}=\sum_{j=1}^{\infty}\left(1-\left|\lambda_{j}\right|^{2}\right)^{2 \sigma}
\backsimeq\sum_{m=0}^{\infty}2^{-2\sigma m}N_m=\infty.$$
Thus,
$$\p\{\mathcal{C}(B_{2}^{\sigma}\left(\mathbb{B}_{d}\right)) \}=0.$$
\end{proof}
For a sequence $\{z_j\}$, if $\inf _{i \neq j} \beta\left(z_{i}, z_{j}\right)>0$, call $\{z_j\}$ weakly separated.
On the unit ball, denote
$$\mathcal{W}(\mathbb{B}_{d}):=\{\omega: \Lambda(\omega) \text{ is weakly separated in $\mathbb{B}_{d}$}\}.$$

We need to point out that the weak separation with respect to Bergman metric is equivalent to the weak separation with respect to pseudo-hyperbolic metric. Thus, we have following lemma.
\begin{lemma}[\protect{\cite[Lemma 3.5]{Massaneda}}]\label{separation}
Let $\Lambda(\omega)=\left\{\lambda_{j}\right\}$ be a random sequence. Then the following statements hold.
\begin{itemize}
\item[(i)] If $$\sum_{m}2^{-dm}N^2_m<\infty,$$then $\p\{\mathcal{W}(\mathbb{B}_{d}) \}=1$.
\item[(ii)] If $$\sum_{m}2^{-dm}N^2_m=\infty,$$ then $\p\{\mathcal{W}(\mathbb{B}_{d}) \}=0$.
\end{itemize}
\end{lemma}

By Theorem \ref{interpolation}, a sequence is an interpolating sequence if and only if it is weakly separated and the corresponding measure is a Carleson measure. The proof of Theorem \ref{interpolation:besove} is now given.

\begin{proof}[Proof of  Theorem \ref{interpolation:besove}]
Since $\frac{1}{2}\geq\sigma>0 $ and $d\geq2$, then $-d+4\sigma\leq0$.
If $\displaystyle\sum_{m=0}^{\infty}2^{-2\sigma m}N_m<\infty$, then $\displaystyle\sum_{m=0}^{\infty}2^{-4\sigma m}N^2_m<\infty$, which implies
$$\sum_{m=0}^{\infty}2^{-dm}N^2_m=\sum_{m=0}^{\infty}2^{(-d+4\sigma)m}2^{-4\sigma m}N^2_m\leq\sum_{m=0}^{\infty}2^{-4\sigma m}N^2_m<\infty.$$
By Theorem \ref{interpolation} and Lemma \ref{separation}, the conclusion follows.

On the other hand, if $\sum_{m=0}^{\infty}2^{-2\sigma m}N_m=\infty$, then
$$\p\{\mathcal{C}(B_{2}^{\sigma}\left(\mathbb{B}_{d}\right)) \}=0.$$
By Theorem \ref{interpolation} again, it follows
$$\p\{\mathcal{I}(B_{2}^{\sigma}\left(\mathbb{B}_{d}\right))\}=0.$$
\end{proof}

Finally, the uniformly separated sequences on the unit ball when $d\geq2$ are studied. 
It is well known that
$$
1-\left|\varphi_{z}(w)\right|^{2}=\frac{\left(1-|z|^{2}\right)\left(1-|w|^{2}\right)}{|1-\langle w, z\rangle|^{2}}.
$$
A sequence $\{z_j\}$ is uniformly separated if $\inf_{k}\prod_{j\neq k}\rho(z_{j},z_k)>0,$ where $\rho$ is the pseudo-hyperbolic distance on the unit ball.
Let
$$\mathcal{U}(\mathbb{B}_{d}):=\{\omega: \Lambda(\omega) \text{ is uniformly separated in $\mathbb{B}_{d}$}\}.$$
An important lemma in the analysis, see \cite[Proposition 1.4.10]{Rudin}, is needed.
\begin{lemma}\label{integral}
For $z \in \mathbb{B}_d$ and $c\in\mathbb{R}$, let
$$
I_{c}(z)=\int_{\partial\mathbb{B}_d} \frac{1}{|1-\langle z, \zeta\rangle|^{d+c}}d \sigma(\zeta).
$$
When $c<0,$ then $I_{c}$ is bounded in $\mathbb{B}_d$.  When $c>0,$ then
$$
I_{c}(z) \approx\left(1-|z|^{2}\right)^{-c}.
$$
Finally,
$$
I_{0}(z) \approx \log \frac{1}{1-|z|^{2}}.
$$
\end{lemma}

\begin{prop}Let $\Lambda(\omega)=\left\{\lambda_{j}\right\}$ be a random sequence. Then the following statements hold.\\
\begin{itemize}
\item[(i)]If $d=2$ and
$$\sum_{m=0}^{\infty}N_m 2^{-m}(m+1)<\infty,$$
then $\p\{\mathcal{U}(\mathbb{B}_{d})\}=1$.
\item[(ii)]If $d\geq3$ and
$$\sum_{m=0}^{\infty}N_m2^{-m}<\infty,$$
 then $\p\{\mathcal{U}(\mathbb{B}_{d})\}=1$.
\item[(iii)]If $d\geq3$ and $$\sum_{m=0}^{\infty}N_m2^{-m}=\infty,$$
then $\p\{\mathcal{U}(\mathbb{B}_{d})\}=0$.
\end{itemize}
\end{prop}

\begin{proof}
First, $\inf_{k}\prod_{j\neq k}\rho(\lambda_{j},\lambda_k)^2>0$ almost surely if and only if
$$\sup_{k}\sum_{j\neq k}-\log{\rho(\lambda_{j},\lambda_k)^2}<\infty.$$
Since $-\log{x}\geq 1-x$ when $1\geq x>0$, if follows
$$-\log{\rho(\lambda_{j},\lambda_k)^2}\geq 1-\rho(\lambda_j,\lambda_k)^2.$$
Thus $\inf_{k}\prod_{j\neq k}\rho(\lambda_{j},\lambda_k)^2>0$ almost surely implies $\sup_{k}\sum_{j\neq k}[1-\rho(\lambda_j,\lambda_k)^2]<\infty$ almost surely.

On the other hand, if  $\inf_{\lambda_j\neq \lambda_k}\rho(\lambda_j,\lambda_k)>0 $, then
$$-\log{\rho(\lambda_{j},\lambda_k)^2}\lesssim1-\rho(\lambda_j,\lambda_k)^2.$$
Thus, in this case, $\sup_{k}\sum_{j\neq k}[1-\rho(\lambda_j,\lambda_k)^2]<\infty$ almost surely implies $\inf_{k}\prod_{j\neq k}\rho(\lambda_{j},\lambda_k)^2>0$ almost surely.

For any constant $c$, consider
\begin{align*}
&\sum_{k=1}^{\infty}\p\left(\left\{\omega: \sum_{j\neq k}[1-\rho(\lambda_j,\lambda_k)^2]>c\right\}\right)
\leq \frac{1}{c}\sum_{k=1}^{\infty}\E\Big[ \sum_{j\neq k}[1-\rho(\lambda_j,\lambda_k)^2]\Big]\\
=&\frac{1}{c}\sum_{k=1}^{\infty}\sum_{j\neq k}\E\Big[[1-\rho(\lambda_j,\lambda_k)^2]\Big]=\frac{1}{c}\sum_{k=1}^{\infty}\sum_{j\neq k}\E\Big[\frac{\left(1-|\lambda_j|^{2}\right)\left(1-|\lambda_k|^{2}\right)}{|1-\langle \lambda_k, \lambda_j\rangle|^{2}}\Big]\\
=& \frac{1}{c}\sum_{k=1}^{\infty}\sum_{j\neq k}
\int_{\partial \mathbb{B}_d}\int_{\partial \mathbb{B}_d}\frac{\left(1-|\lambda_j|^{2}\right)\left(1-|\lambda_k|^{2}\right)}{|1-\langle \lambda_k, \lambda_j\rangle|^{2}}d\sigma(\xi_{j})d\sigma(\xi_{k})\\
\leq &\frac{1}{c}\sum_{k=1}^{\infty}\sum_{j\neq k}
\left(1-|\lambda_j|^{2}\right)\left(1-|\lambda_k|^{2}\right)\int_{\partial \mathbb{B}_d}\int_{\partial \mathbb{B}_d}\frac{1}{|1-\langle |\lambda_j||\lambda_k|\xi_{k}, \xi_j\rangle|^{d+2-d}}d\sigma(\xi_{j})d\sigma(\xi_{k}).
\end{align*}
Next, consider two cases.

If $d=2$, then by Lemma \ref{integral} we have
$$
\int_{\partial \mathbb{B}_d}\frac{1}{|1-\langle |\lambda_j||\lambda_k|\xi_{k}, \xi_j\rangle|^{d+2-d}}d\sigma(\xi_{j})\lesssim\log \frac{1}{1-|\lambda_j|^2|\lambda_k|^2}.
$$
Substituting in this estimate to the above yields that
\begin{align*}
&\sum_{k=1}^{\infty}\p\left(\left\{\omega: \sum_{j\neq k}[1-\rho(\lambda_j,\lambda_k)^2]>c\right\}
\right)
\lesssim  \frac{1}{c}\sum_{k=1}^{\infty}
\sum_{j=1}^{\infty}\left(1-|\lambda_j|^{2}\right)\left(1-|\lambda_k|^{2}\right)\log\frac{1}{1-|\lambda_j|^2|\lambda_k|^2}\\
\lesssim&\sum_{l=0}^{\infty}\sum_{m=0}^{\infty}N_lN_m 2^{-l}2^{-m}\Big[\log{(\frac{1}{2^{-l-1}+2^{-m-1}})}+1\Big]\\
\leq&2\sum_{l\geq m}N_lN_m 2^{-l}2^{-m}\Big[\log{(\frac{1}{2^{-l-1}+2^{-m-1}})}+1\Big]\lesssim\sum_{l\geq m}N_lN_m 2^{-l}2^{-m}(l+1)\\
=&\sum_{m=0}^{\infty}N_m 2^{-m} \sum_{l=m}^{\infty}N_l2^{-l}(l+1)\leq \left(\sum_{l=0}^{\infty}N_l 2^{-l}(l+1)\right)^{2}.
\end{align*}
Thus $\sum_{m=0}^{\infty}N_l 2^{-l}(l+1)<\infty$ implies that
$$\sup_{k}\sum_{j\neq k}[1-\rho(\lambda_j,\lambda_k)^2]<\infty$$
almost surely.
Lemma \ref{separation} also yields that $\sum_{m=0}^{\infty}N_l 2^{-l}(l+1)<\infty$ implies that $\Lambda(\omega)$ is weakly separated almost surely, thus $\sum_{m=0}^{\infty}N_l 2^{-l}(l+1)<\infty$ implies that $\Lambda(\omega)$ is uniformly separated almost surely.

If $d\geq3$, by Lemma \ref{integral}, then
$$
\int_{\partial \mathbb{B}^d}\frac{1}{|1-\langle |\lambda_j||\lambda_k|\xi_{k}, \xi_j\rangle|^{d+2-d}}d\sigma(\xi_{j})\leq C_d<\infty.
$$
Then
\begin{align*}
\sum_{k=1}^{\infty}\p\left(\left\{\omega: \sum_{j\neq k}[1-\rho(\lambda_j,\lambda_k)^2]>c\right\}\right)
\lesssim\sum_{k=1}^{\infty}\sum_{j=1}^{\infty}\left(1-|\lambda_j|^{2}\right)\left(1-|\lambda_k|^{2}\right)
\backsimeq\left(\sum_{j=1}^{\infty}N_{m}2^{-m}\right)^2.
\end{align*}
By a similar argument for the case of $d=2$, $\sum_{j=1}^{\infty}N_{m}2^{-m}<\infty$ implies that $\Lambda(\omega)$ is separated almost surely.
Conversely, if
$\sum_{j=1}^{\infty}N_{m}2^{-m}=\infty$, then  $\sum_{j=1}^{\infty}\left(1-|\lambda_j|^{2}\right)=\infty.$
Thus, for any $z_k$, there holds
\begin{align*}
\sum_{j\neq k}-\log{\rho(\lambda_{j},\lambda_k)^2}&\geq \sum_{j\neq k}1-\rho(\lambda_j,\lambda_k)^2=
\sum_{j\neq k}\frac{(1-|\lambda_k|^{2})(1-|\lambda_j|^{2})}{|1-\langle \lambda_j,\lambda_k\rangle|^2}\\
&\geq\frac{(1-|\lambda_k|^{2})}{(1+|\lambda_k|)^2}\sum_{j\neq k}(1-|\lambda_j|^{2})=\infty.
\end{align*}
Giving the conclusion that
$$\p\{\mathcal{U}(\mathbb{B}_{d}) \}=0.$$
\end{proof}

\end{document}